\documentclass[11pt]{article}
\usepackage{hyperref}
\usepackage{latexsym}
\usepackage{amsmath,amsthm,amsfonts,amssymb}
\usepackage{bbm}
\usepackage[T1]{fontenc}
\usepackage[utf8]{inputenc}
\usepackage{url}
\usepackage{graphicx}
\usepackage{color}
\usepackage[a4paper,textwidth=15cm,textheight=25cm]{geometry}
\usepackage[francais,english]{babel}

\newtheorem{thm}{Theorem}[section]
\newtheorem{thmbis}{Theorem}
\newtheorem*{thm*}{Theorem}
\newtheorem{dfn}[thm]{Definition} 
\newtheorem*{dfn*}{Definition}

\newtheorem{cor}[thm]{Corollary}
\newtheorem*{cor*}{Corollary}

\newtheorem{prop}[thm]{Proposition} 
\newtheorem*{prop*}{Proposition} 
\newtheorem*{properties*}{Properties} 
\newtheorem{propbis}[thmbis]{Proposition} 
\newtheorem{lem}[thm]{Lemma} 
\newtheorem*{lem*}{Lemma}

\newtheorem*{claim*}{Claim} 
 
\newtheorem*{fact*}{Fact} 
\newtheorem{fact}[thm]{Fact} 
 
\newtheorem*{qst*}{Question}

\newtheorem*{pb*}{Problem}

\newtheorem*{conj*}{Conjecture}

\theoremstyle{remark}
 
\newtheorem*{algo*}{Algorithm} 
\newtheorem*{rem*}{Remark}
\newtheorem{rem}[thm]{Remark}
\newtheorem*{example*}{Example}

\newcounter{numEnonceTmpInterne}
\newenvironment{enonce*}[1]{\theoremstyle{plain}\stepcounter{numEnonceTmpInterne}%
\def\a{enoncetmp\alph{numEnonceTmpInterne}}%
\newtheorem*{\a}{#1}\begin{\a}}{\end{\a}}

\makeatletter
\edef\@tempa#1#2{\def#1{\mathaccent\string"\noexpand\accentclass@#2 }}
\@tempa\rond{017}
\makeatother

\newcommand{\es}{\emptyset}
\renewcommand{\phi}{\varphi} 
\newcommand{\m} {^{-1}} 
\newcommand{\eps} {\varepsilon}

\newcommand {\ra} {\rightarrow}

\newcommand {\onto} {\twoheadrightarrow}
\newcommand {\into} {\hookrightarrow}

\newcommand{\ol}[1]{\overline{#1}} 

\newcommand{\semidirect}{\ltimes}
\newcommand{\isemidirect}{\rtimes}
\newcommand{\normal} {\vartriangleleft}
\renewcommand{\subsetneq}{\varsubsetneq}

\newcommand{\dunion}{\sqcup}

\newcommand{\ie} {i.~e.\ }

\newcommand {\cala} {{\mathcal {A}}}   
   
\newcommand {\calc} {{\mathcal {C}}}   
\newcommand {\cald} {{\mathcal {D}}}   
\newcommand {\cale} {{\mathcal {E}}}   
\newcommand {\calf} {{\mathcal {F}}}

\newcommand {\calj} {{\mathcal {J}}}

\newcommand {\cals} {{\mathcal {S}}}

\newcommand {\bbN} {{\mathbb {N}}}

\newcommand {\bbQ} {{\mathbb {Q}}}   
\newcommand {\bbR} {{\mathbb {R}}}

\newcommand {\bbZ} {{\mathbb {Z}}}   

\newcommand{\grp}[1]{\langle #1 \rangle}

\newcommand{\Stab} {{\mathrm{Stab}}}

\newcommand{\Disc}{\mathrm{Disc}}
\newcommand{\Reg}{\mathrm{Reg}}
\newcommand{\Sing}{\mathrm{Sing}}
\newcommand{\Irred}{\mathrm{Irred}}
\newcommand{\fin}{\mathrm{fin}}

\newcommand{\gobble}[1]{} 

\usepackage{srcltx}

\newcommand{\IET}{\mathrm{IET}}

\begin{document}

\title{Solvable groups of interval exchange transformations}
\author{Fran\c{c}ois Dahmani, Koji Fujiwara, Vincent Guirardel}

\date{\today.}

\maketitle
\selectlanguage{english}
\begin{abstract}
We prove that  any finitely generated torsion free solvable subgroup 
of the group $\IET$ of
 all Interval Exchange Transformations is virtually abelian. 
 In contrast, 
the lamplighter groups $A\wr \bbZ^k$ embed in $\IET$ for every finite
abelian group $A$, 
and we construct uncountably many non pairwise isomorphic  3-step solvable subgroups
 of $\IET$ as semi-direct products of a lamplighter group with an abelian group. 

We also prove that for every non-abelian finite group 
$F$,  the group $F\wr\bbZ$ 
does not embed in $\IET$. 
\end{abstract}

\selectlanguage{francais}
\begin{abstract}
  Nous démontrons que tout sous-groupe de type fini résoluble du groupe $\IET$ des échanges d'intervalles
est virtuellement abélien. A l'opposé, les groupes d'allumeurs de réverbères $A\wr \bbZ^k$
se plongent dans $\IET$ pour tout groupe abélien fini $A$,
et nous construisons un nombre non dénombrable de sous-groupes résolubles de classe 3 dans $\IET$ non isomorphes entre eux
comme produits semi-directs d'un groupe d'allumeur de réverbère avec des groupes abéliens.

Nous démontrons aussi que pour tout groupe fini non-abélien $F$, le produit en couronne $F\wr \bbZ$ ne se plonge pas dans $\IET$.
\end{abstract}

\selectlanguage{english}

\subsection*{The group $\IET$ and its subgroups}
The group $\IET$ of interval exchange transformations is the group of
all bijections of the interval $[0,1)$ that are piecewise translations
with finitely many discontinuity points.  

Rather unexpectedly, the
recent study of this group has given evidences that it is not as big
as one could have thought, in several commonly accepted features. For
instance, in \cite{DFG1} we established that $\IET$ 
does not have  many free subgroups (if any at all), and that the connected Lie groups
that can embed in it are only abelian. 
Another fact in this direction is that given any finitely generated
subgroup of $\IET$, and any point $x\in [0,1)$ the orbit of $x$ grows
(in cardinality) at most polynomially in the word length of the
elements of the subgroup \cite[Lemma 6.2]{DFG1}.

Yet another instance of these evidences is given by   the main result
of \cite{JM} of  Juschenko and Monod which  implies that
certain natural subgroups of this group are amenable. 
More precisely, given $\alpha\in \bbR\setminus \bbQ$, 
the subgroup  $\IET_{\alpha}$ of transformations whose translation lengths are all multiples of $\alpha$ modulo $1$
is amenable. Indeed, given any finitely generated subgroup $G$ of $\IET_{\alpha}$, $G$ can be viewed as a group of
homeomorphisms of the Cantor set $K$ obtained by blowing up the $R_\alpha$-orbit of the discontinuity points of the generators of $G$,
where $R_\alpha$ is the rotation $x\mapsto x+\alpha \mod 1$ \cite{Cornulier_Bourbaki}.
Denoting by $\hat R_\alpha$ the homeomorphism of $K$ induced by $R_\alpha$, 
this embeds $G$ in the full topological group of $\hat R_\alpha$,
 which is amenable by \cite{JM}.  
This has been extended in \cite{JMMS}   
to subgroups of rational rank $\leq 2$,
ie such that the  subgroup of $\bbQ/\bbZ$ generated by the translation
lengths of its elements does not contain $\bbZ^3$.

Given these evidences,  we chose to investigate  the possible solvable
subgroups of $\IET$.

\subsection*{Results}

In order to describe elementary examples of subgroups of $\IET$, let
us enlarge a bit the 
context, and instead of interval exchange transformations on the
interval $[0,1)$ we consider the group $\IET(\cald)$ of interval  
exchange transformations on a domain $\cald$
consisting of a disjoint union of
finitely many oriented circles, and oriented half-open intervals,
closed on the left (see Section \ref{S1-def}). 
 This does not make change the isomorphism classes of subgroups encountered as $\IET\simeq \IET(\cald)$.

For all $n\in \bbN$, $\bbZ^n$ embeds in $\IET(\bbR/\bbZ)$ as a group of rotations.
The following general simple fact then implies that every finitely generated virtually abelian group embeds in $\IET$.

\begin{propbis}[Proposition \ref{prop_ext}]
  Let $G$ be a group, and assume that some finite index subgroup of $G$ embeds in $\IET$.
Then so does $G$.
\end{propbis}

It is then natural to ask which virtually polycyclic groups embed in $\IET$.
Our first result shows that only virtually abelian ones do.
By a different method, Cornulier \cite{Cornulier} also shows that a virtually polycyclic group
in $\IET$ must be virtually abelian.

\begin{thmbis}[See Corollary \ref{cor;VPC}]
  Let $H$ be a virtually polycyclic group. Then $H$ embeds into $\IET$ if and only if it virtually abelian.
\end{thmbis}

 Since a polycyclic group is virtually torsion-free,  this result is in fact a corollary of the following theorem which applies to
all torsion-free solvable subgroups of $\IET$.

\begin{thmbis} [Theorem \ref{thm;tf_vA}]\label{thm_solv_intro}
 Every  finitely generated torsion-free solvable subgroup of $\IET$ is
virtually abelian.  
\end{thmbis}

If we allow torsion, a much greater variety of subgroups exists. 
The first interesting example is an embedding of the lamplighter group  $L=(\bbZ/n\bbZ) \wr \bbZ^k$ in $\IET$.
Note that this group $L$ is solvable (in fact metabelian), has exponential growth, and is not virtually torsion-free.

To describe this embedding, consider the domain  $\cald=(\bbZ/n\bbZ)\times (\bbR/\bbZ)$, a disjoint union
of $n$ circles. Choose $\Lambda\subset \bbR/\bbZ$ a subgroup isomorphic to $\bbZ^k$, and view $\Lambda$ as a group of \emph{synchronized rotations} in $\IET(\cald)$, \ie by making $\theta\in\Lambda$ act on 
$\cald=(\bbZ/n\bbZ)\times (\bbR/\bbZ)$ by $(i,x)\mapsto (i,x+\theta)$.
Consider the interval $I=[0,1/2) \subset \bbR/\bbZ$, and let $\tau$ be
the transformation that is the identity outside $(\bbZ/n\bbZ)\times I$,
and sends $(i,x)$ to $(i+1,x)$ for $x\in I$.
Then the subgroup of $\IET(\cald)$ generated by $\tau$ and $\Lambda$
is isomorphic to $L$. This is illustrated in Figure \ref{fig;1}.  
More generally,  taking $\cald=A\times (\bbR/\bbZ)$ for some finite abelian group $A$, 
this construction  yields  
the following result: 

\begin{propbis}[see Propositions \ref{prop_abelianLL} and \ref{prop;LLk}]
  For any finite abelian group $A$ and any $k\geq 1$, the wreath product
$L=A \wr \bbZ^k$ embeds in $\IET$.
\end{propbis}

One could try a similar construction, replacing the abelian finite group $A$ by a non-abelian one.
But the group obtained would not be the wreath product. In fact, such a wreath product
cannot embed in $\IET$
as the following result shows.

\begin{thmbis}[Theorem \ref{thm_LLNA}]  If $F$ is a finite group and if $F\wr \bbZ$ embeds as a
  subgroup in $\IET$, then $F$ is abelian.
\end{thmbis}

 It would be
interesting to know which finitely generated wreath products embed in
$\IET$.

Starting from a subgroup $G<\IET(\cald)$ and a finite abelian group $A$, the construction above allows to
construct subgroups of $\IET(\cald\times A)$ isomorphic to $G\semidirect \calf_A$ where $\calf_A$ is a subgroup of
the abelian group $A^{\cald}$. 
We then prove that, in contrast with the torsion-free case, this construction yields a huge variety of isomorphism
classes of solvable subgroups in $\IET$.

\begin{thmbis}[Theorem \ref{thm_uncountable}]\label{thm_uncountable_intro}
There exists uncountably many  isomorphism classes of subgroups of
$\IET$ that are generated  by $3$ elements, and that are solvable of derived
length $3$.
\end{thmbis}

The method we use consists of  embedding many semidirect products in a
way that is related to the twisted embeddings used in 
\cite{JMMS}.

\subsection*{About proofs}

The proof of Theorem \ref{thm_LLNA} saying that $F\wr \bbZ$ does not embed into $\IET$ uses the fact that orbits in $[0,1)$ by a finitely generated subgroup of $\IET$ have
 polynomial growth. 
On the other hand, if $F\wr \bbZ=(\oplus_{n\in \bbZ} F)\isemidirect \bbZ$ embeds in $\IET$, $F$ and its conjugate have to commute
with each other. This gives strong algebraic restrictions on the action on $F$.
Using Birkhoff theorem, we show that if $F$ itself is non-commutative, then the orbit growth of $F\wr \bbZ$ has to be exponential,
a contradiction.

To prove that there are uncountably many groups as in Theorem
 \ref{thm_uncountable_intro}, 
we start with a lamplighter group $G = (\bbZ/3\bbZ) \wr \bbZ$ constructed above on $\cald=(\bbZ/3\bbZ)\times (\bbR/\bbZ)$,
where $\bbZ$ acts on the three circles by setting a generator to act as a
synchronized rotation on them 
 with irrational angle $\alpha$. 
Then, we consider $\cald'=(\bbZ/2\bbZ)\times \cald$, on which we make
$G$ act diagonally.
Then we choose an interval $J$ in $\{0\}\times (\bbR/\bbZ) \subset
\cald$ and define $\tau_J$ on $\cald'$ by the identity
outside $(\bbZ/2\bbZ)\times J$, and by $(i,x)\mapsto (i+1,x)$ on
$(\bbZ/3\bbZ)\times J\subset \cald'$. This is illustrated in Figure \ref{fig;2}.
The group $H$ generated by $G$ and $\tau_J$ is a homomorphic image of
the wreath product $(\bbZ/2\bbZ) \wr G$, but this is not an
embedding.  
Still, the group $H$ generated by $G$ and $\tau_J$ has the structure
of a semidirect product $G\semidirect \calf$ where $\calf$ is an
infinite abelian group of exponent $2$. 
 Using Birkhoff ergodic theorem, we prove that from the isomorphism class of $H$, one can read  off the length of $J$ modulo a countable additive group.
This proves that by  varying length of $J$, we get uncountably many distinct isomorphism classes of groups $H$.

The classification of torsion-free solvable subgroups of $\IET$
  in Theorem \ref{thm_solv_intro} is based
on the fact that centralizers of a \emph{minimal} interval exchange transformation $T$ is small.
Indeed, if $T$ is an irrational rotation on a circle, then its centralizer consists of the whole group of rotations on this circle;
and if $T$ is not conjugate to such a rotation, a theorem by Novak \cite{No} shows that its centralizer is virtually cyclic.  
If $T$ is not minimal, its centralizer can be much larger: it will for instance contain a group isomorphic to $\IET$ if
$T$ fixes a non-empty subinterval.

Thus, we need to
understand the orbit closures of a finitely generated group, and we also need 
to understand how it varies when we pass to a subgroup of finite index.

A result by Imanishi \cite{Ima} about the holonomy of codimension 1 foliations shows that for each finitely generated group
$G<\IET(\cald)$, there is a partition of $\cald$ 
into finitely many $G$-invariant subdomains\footnote{A subdomain of $\cald$ is a subset that consists of finitely many semi-open intervals, closed on the left},
such that in restriction to each subdomain, either every orbit is dense (such a subdomain is called an \emph{irreducible component}), 
or every orbit is finite of the same cardinal (and one can say more,
see Proposition \ref{prop_Imanishi}). 
In particular, for each $x\in\cald$, $\overline{G.x}$ is either finite, or is the closure of a subdomain (not a Cantor set).

When passing to a finite index subgroup $G_0$ of $G$, it could happen that an irreducible component for $G$ splits into several 
irreducible  components for $G_0$.

For example, consider $\cald=(\bbZ/2\bbZ)\times(\bbR/\bbZ)$, and consider the subgroup $G$ of $\IET(\cald)$ generated by the three following tranformations $\tau,R_0,R_1$.
Let $\tau$ be the involution $(i,x)\mapsto (i+1,x)$;
let $\alpha\in\bbR\setminus\bbQ$ 
and let $R_0$ be the rotation of angle $\alpha$ 
on the circle $\{0\}\times (\bbR/\bbZ)$ 
and as the identity on the circle $\{1\}\times (\bbR/\bbZ)$;
and let $R_1=\tau R_0\tau\m$.
Then  $\grp{R_0,R_1}\simeq\bbZ^2$, and $G=\grp{\tau,R_0}\simeq (\bbZ \oplus \bbZ) \isemidirect (\bbZ/2\bbZ)$,
and  any orbit of $G$ is dense.  

However,  $G_0=\grp{R_0,R_1}$ is a finite index
subgroup which preserves each circle. 
 The two circles are the irreducible components of $G_0$.  But this phenomenon cannot occur any more when 
passing to a further finite index $G_1$ of $G_0$ because
 $G_1$ has to contain an irrationnal rotation on each
circle, thus ensuring that 
the two circles are still irreducible components of $G_1$. 
This group $G_0$ is what we call \emph{unfragmentable}.
The following technical result  of independent interest shows that this is a general fact.

\begin{thmbis}[see Theorem \ref{thm;fi_unfragmentable}]
  Given a finitely generated group $G<\IET$, there exists a finite index subgroup $G_0$ which is \emph{unfragmentable} in the following sense:
  \begin{enumerate}
  \item $G_0$ acts by the identity outside its irreducible components
  \item If $G_1$ is a finite index of $G_0$, any irreducible component
    of $G_0$ is also an irreducible component of $G_1$.
  \end{enumerate} 
\end{thmbis}

\section{Generalities}

\subsection{Definitions}
\label{S1-def}

\begin{dfn}
In all the following, a \emph{domain} will be a non-empty disjoint union of
finitely many oriented circles, and oriented half-open, bounded
intervals (closed on the left). 

Given a domain $\cald$, the group  $\IET(\cald)$  of interval exchange transformations
of $\cald$ is the group of bijections of $\cald$ that are orientation
preserving piecewise isometries, left continuous with finitely many 
discontinuity points.
\end{dfn}

By convention, we define $\IET= \IET([0,1))$.

Given two disjoint domains $\cald_1, \cald_2$ having the same total length,
there is an element of
$\IET(\cald_1 \sqcup \cald_2)$ that sends $\cald_1$ on
$\cald_2$.   We
call such an element an interval exchange bijection from $\cald_1$ to
$\cald_2$. This element  
then conjugates $\IET(\cald_2)$ to $\IET(\cald_1)$.  Observe also that
rescaling a domain doesn't change its group of interval exchange
transformations. In particular,  for any domains $\cald_1, \cald_2$, 
$\IET(\cald_1)$ and $\IET(\cald_2)$    
  are always isomorphic.

A subdomain $\cald_0\subset \cald$ is a subset of $\cald$ which
has finitely many connected components, and which is closed on the left.
If $\cald_0$ is a subdomain  of $\cald$ that is invariant by the action
of $G$, then $G$ naturally maps  to $\IET(\cald_0)$ by restriction. We denote by $G_{|\cald_0}$
its image in $\IET(\cald_0)$.
Moreover, if
$\cald = \cald_0 \dunion \cald_1$ where both subdomains are invariant
by $G$, then the induced morphism from $G$ to $\IET(\cald_0) \times
\IET(\cald_1)$ is injective.

\subsection{Finite extensions}
\begin{prop}\label{prop_ext}
  Let $G$ be a group, and assume that some finite index subgroup of $G$ embeds in $\IET$.
Then so does $G$.
\end{prop}

\begin{proof}
Without loss of generality, consider $H<G$ a normal subgroup of finite index that embeds in $\IET(\cald)$ for some domain $\cald$.
Let $Q$ be the finite quotient $Q=G/H$.
It is a classical algebraic fact (see \cite{KK}) that $G$ embeds in
the wreath product $H\wr Q=H^Q\isemidirect Q$ (where $Q$ acts on $H^Q$
by permuting coordinates). 
Thus, it suffices to show that $H^Q\isemidirect Q$ embeds in $\IET$.

Consider the domain $\cald'=Q\times \cald$, and embed $H^Q$ in $\IET(\cald')$ by
making $(h_q)_{q\in Q}$ act on $\cald'=Q\times\cald$ by  $(q,x)\mapsto (q,h_q.x)$.
Then $Q$ acts on $\cald'$ by left multiplication on the left coordinate.
This naturally extends to a morphism $H^Q\isemidirect Q\ra
\IET(\cald')$ which is clearly  one-to-one. 
\end{proof}

\section{Irreducibility and unfragmentability for finitely generated subgroups
of $\IET$}

\subsection{$\IET$  and irreducibility} \label{S1-irred}

Let $\cald$ be a domain.  
Let $G=\langle S\rangle$ be a finitely generated subgroup of the
group $\IET(\cald)$ of interval exchange transformations on $\cald$,  with $S$ symmetric.

\begin{dfn}[Irreducibility.]
We say that $G$ is \emph{irreducible} (on $\cald$)  
if no subdomain of $\cald$ is invariant under $G$. 

We say that  a subdomain $J$ 
of $\cald$ is an \emph{irreducible component}
for $G$ if it is $G$-invariant, and if $G$ restricted to $J$ is irreducible.
\end{dfn}

We will see that a finitely generated group $G$ is irreducible if and only if every $G$-orbit is dense in $\cald$ (see Corollary \ref{cor_irred}).

If $s\in \IET(\cald)$, we denote by  $\Disc(s)\subset \rond\cald$ the set of discontinuity
points of $s$. If $S$ is a set of elements of $\IET(\cald)$, 
 we denote by $\Disc(S)=\bigcup_{s\in S} \Disc(s)$ the set of discontinuity points of
elements of $S$.

We say that $x,y\in\rond{\cald}$ are in the same 
\emph{regular orbit} if there exist $g\in G$ continuous at $x$ with $g(x)=y$.
We say that $x,y\in\rond{\cald}$ are in the same \emph{$S$-regular orbit}
if there exists  $g_1,\dots ,g_n\in S$ such that $y=g_n\dots g_1(x)$, and 
for all $i\leq n$, $g_{i-1}\dots g_1(x)\in\rond \cald\setminus \Disc(g_i)$.
We denote by $\Reg(x,G)\subset \rond \cald$ the regular orbit of $x$, and by $\Reg(x,S)\subset \Reg(x,G)$ its $S$-regular orbit.
Although arguably less natural, the notion of $S$-regular orbit is the one that is needed to apply Imanishi Theorem below.
On the other hand, Lemma \ref{lem_reg} below will show that one can choose $S$ so that 
$\Reg(x,S)= \Reg(x,G)$ for all $x\in \rond \cald$.

Let $$\Sing(G)=\{x\in \rond D|\exists g\in G,\ x\in \Disc(g)\}.$$
By definition, $\Sing(G)$ does not depend on any generating set,
but one easily checks that for $x\in \rond \cald$, $x\in \Sing(G)$ if and only if
its $G$-orbit contains a point in $\Disc(S)$, if and only if 
$\Reg(x,S)$ contains a point in $\Disc(S)$.
Thus, $\Sing(G)$ is a union of at most $\#\Disc(S)$ $S$-regular orbits.
We also note that for all $x\in\rond\cald\setminus\Sing(G)$, $\Reg(x,S)=\Reg(x,G)=G.x\cap \rond \cald$.

We denote by $E(G)\subset \Sing(G)$ the set of points $x\in \Sing(G)$ whose regular orbit is finite,
and by $E(S)$ the set of points $x\in \Sing(G)$ whose $S$-regular orbit is finite.
We note that $E(G)\subset E(S)$ and that these sets are finite since
$$E(S)=\bigcup_{x\in\Disc(S),\#\Reg(x,S)<\infty} \Reg(x,S)$$
is a finite union of finite sets.

 We will need to apply a decomposition theorem that applies for finite systems of isometries, so let us introduce the corresponding terminology (see \cite{GLP1}).
Let $\ol\cald$ be the obvious compactification of $\cald$ as a union of compact intervals and circles. A partial isometry on $\ol \cald$ is
an isometry $\phi:I\ra J$ between two closed subintervals $I,J\subset \cald$.
Given a finite set $S$ of interval exchanges on $\cald$,
one can construct a \emph{system of isometries} on $\ol\cald$ as follows:
for each $s\in S$, let $I_1,\dots,I_{n_s}$ be the maximal connected subdomains on which $s$ is continuous.
These intervals define a partition of $\cald$. For each $i\leq n_s$, we let $\phi_i$ be the partial isometry
defined on $\ol I_i$ that extends $s_{|I_i}$. 
Thus, for each element of $s$, we have a finite collection of partial isometries,
and we denote by $X$ the collection of all partial isometries of $\ol \cald$ obtained from all the elements of $S$ in this way.
Orbits of $X$ are defined in the natural way. Clearly, each $G$-orbit is contained in an $X$-orbit.
One defines $\rond{X}$-orbits similarly using the restriction of all the partial isometries to the interior of their domains.
Then for all $x\in\rond D$, its $\rond{X}$-orbit coincides exactly with its $S$-regular orbit $\Reg(x,S)$.

Specifying \cite[Th. 3.1]{GLP1} to our setting, we get:

\begin{prop}[Imanishi  theorem \cite{Ima}, see Th. 3.1 in \cite{GLP1}]\label{prop_Imanishi}
Let $S$ be a finite symmetric set of interval exchanges on $\cald$, and $G=\grp{S}<\IET(\cald)$. 

Then $\rond\cald\setminus E(S)=\rond\cald_1\dunion\dots \dunion \rond \cald_p$ where each $\rond \cald_i$ is an open subset of $\cald\setminus E(S)$
invariant under $S$-regular orbits
 and such that
for each $i\leq p$, one of the following holds:
\begin{enumerate}
\item 
every $S$-regular orbit in $\rond\cald_i$ is dense in $\rond\cald_i$
\item every $S$-regular orbit in $\rond\cald_i$ is finite, of the same cardinality,
and $\rond \cald_i\cap \Sing(G)=\es$.
\end{enumerate}
\end{prop}

\begin{rem} 
Theorem 3.1 of \cite{GLP1} is stated for systems of isometries on a finite union of intervals, but
generalizes immediately to the case where we allow circles in the domain
(alternatively, we could restrict to this case by cutting $\cald$
along a point in each circle of $\cald$).
Our assertion 2 thus includes the possibility that $\rond\cald_i$ is a disjoint union of circles, and that all
elements of $G$ act continuously on $\cald_i$. 
Since all elements of $\IET$ preserve the orientation, there cannot be any twisted family of finite orbits in the sense of \cite[Th 3.1]{GLP1}.
\end{rem}

\begin{cor}\label{cor_Imanishi}
Let $G<\IET(\cald)$ be a finitely generated group.

Then $\cald$ decomposes into $G$-invariant subdomains $\cald=\cald_{\infty}\dunion \cald_{\fin}$
with $\cald_{\infty}=I_1\dunion \dots\dunion I_r$ and 
$\cald_{\fin}=J_1\dunion \dots\dunion J_t$
where each $I_i$ and $J_j$ is a $G$-invariant subdomain whose boundary is contained in  $E(S)\cup \partial\cald$, 
and such that
\begin{enumerate}
\item for all $i\leq r$, every $G$-orbit in $I_i$ is dense in $I_i$
\item for all $j\leq r$,  $G$ acts on $J_j$ with finite orbits, all of the same cardinality;
the restriction to $J_j$ of any $g\in G$ is continuous, and in
particular, $G$ permutes the connected components of $J_j$.
\end{enumerate}

Moreover, the collection of irreducible 
components of $G$, $\Irred(G)=\{I_1, \dots, I_r\}$, is unique.
\end{cor}

\begin{rem}
Note that whereas $\Irred(G)=\{I_1, \dots, I_r\}$ is  uniquely 
defined, we don't claim  that the decomposition of the complement $\cald \setminus (\bigcup_i I_i) $ into $G$-invariant subdomains upon
which $G$ acts as a finite group is unique  (although one could easily construct such a canonical decomposition). 
\end{rem}

\begin{proof}
Take $S$ a finite symmetric generating set of $G$. Write $\rond \cald\setminus E(S)=\rond \cald_1\dunion\dots\dunion \rond \cald_p$
as in Proposition \ref{prop_Imanishi}.
Let $\cald_i\subset \cald$ be the smallest subdomain of $\cald$ containing $\rond \cald_i$ (i.e.\ the set of points to which are limits of points in $\rond \cald_i$ from the right), so that $\cald=\cald_1\dunion \dots \dunion \cald_p$.
Then $\cald_i$ is $G$-invariant because if $g\cald_i\cap \cald_j\neq \es$, then $\cald_i\cap g\m \cald_j$ contains an interval
and for any point $x\in \rond \cald\setminus \Sing(G)$ in this interval,
we have that $g(x)\in \Reg(x,S)\cap \cald_j$ contradicting that $\rond\cald_i$ is invariant under $S$-regular orbits.
We denote by $I_1,\dots,I_r$, $J_1,\dots,J_t$ the subdomains of $\cald_1,\dots,\cald_p$ (with $r+t=p$) 
according to whether
they satisfy the first or the second assertion of Proposition \ref{prop_Imanishi}

Then for all $i\leq r$ and every $x\in \rond I_i$, its $G$-orbit is dense in $I_i$ because it contains $\Reg(x,S)$.
If $x\in I_i\setminus \rond I_i$, then its $G$-orbit has to contain a point in $\rond I_i$. Indeed, otherwise,
the $G$-orbit of $x$ would be finite, and so would be the $G$-orbits of points $x'$ in a right neighbourhood of $x$,
contradicting that regular orbits are dense in $\rond I_i$.

For $j\leq t$, 
then since $\rond J_j\cap \Sing(G)=\es$, $G$ acts continously on $\rond J_j$, so $g_{|J_j}$ is continuous for every $g\in G$.

The uniqueness of $\{I_1,\dots,I_r\}$ is immediate from the fact that $I_i$
is $G$-invariant and that every $G$-orbit is dense in $I_i$.
\end{proof}

The following corollary is clear.

\begin{cor}\label{cor_irred}
A subgroup $G<\IET(\cald)$ is irreducible if and only if every $G$-orbit is dense in $\cald$.\qed
\end{cor}

\begin{rem}\label{rem;rem} 
In the case where $E(S)=\es$, each subdomain $I_i$, $J_j$ appearing in Corollary \ref{cor_Imanishi} is a union of connected components of $\cald$.
We can always cut the domain $\cald$ to reduce to this situation as follows.

 We  say that we cut  a domain $\cald_1$ along a finite set
of points if we map it to some domain $\cald_2$ by an interval
exchange bijection $\tau: \cald_1 \to  \cald_2$ that is discontinuous
exactly on this set
of points, and such that $\tau^{-1}$ is continuous. For instance, if
$\cald_1 =[a,b)$, cutting along $c\in [a,b)$ yields a domain $\cald_2=
[a,c) \sqcup [c,b)$, and if $\cald_1$ is a circle, cutting along one
point yields a domain that is a half-open interval.
If $G_1\subset \IET(\cald_1)$, we then get by conjugation a group $G_2\subset \IET(\cald_2)$ (in particular, $G_2$ is isomorphic to $G_1$).

Now consider $G<\IET(\cald)$, and $S$ is a symmetric generating set of $G$. 
Then cutting $\cald$ along the (finitely many) points in $E(S)$
yields a domain $\cald'$ and a group $G' <\IET(\cald')$
 conjugate to $G$ by an interval exchange bijection from $\cald$ to $\cald'$
  such that $E(S')=\es$ (where $S'$ is the conjugate of $S$).
\end{rem}

\subsection{Virtual unfragmentability for finitely generated groups of IET}

Let $G$ be a finitely generated subgroup of $\IET(\cald)$.

\begin{dfn}[Unfragmentability] 
We say that $G$ is \emph{unfragmentable} if for any subdomain $J\subset \cald$ which is invariant by a finite index subgroup of $G$,
$J$ is $G$-invariant.
\end{dfn}

We say that an element $a\in \IET(\cald)$ is \emph{unfragmentable} if the cyclic group $\grp{a}$ is unfragmentable.

Here is an equivalent definition.

\begin{lem}
The subgroup $G$ is \emph{unfragmentable} if each of its finite orbits is
trivial, and if, for every irreducible component  $J$ of $G$, and every
finite index subgroup $H$ of $G$,     the restriction of $H$ on
$J$ is  irreducible.      
\end{lem}

\begin{proof}
Let $I_1,\dots,I_r\subset \cald$ be the
  irreducible components of $G$, and
  $I'=\cald\setminus (I_1\cup\dots\cup I_r)$.

Assume that $G$ is unfragmentable.  If $G$ has a non-trivial finite orbit
  (necessarily in $I'$), then there exists a subdomain $J\subset I'$ and $g\in G$
  such that $g J\neq J$.  Then $J$ is a subdomain that is invariant
  under the finite index subgroup of $G$ acting trivially on $I'$, but
  not $G$-invariant, a contradiction.  If some finite index subgroup
  $G_0<G$ does not act irreducibly on some $I_i$, then there is a
  $G_0$-invariant subdomain $J\subsetneq I_i$ and $J$ is not
  $G$-invariant because $I_i$ is an irreducible component of $G$, a
  contradiction.  

Conversely, assume that the statement in the lemma holds
  and let $J\subset \cald$ be a subdomain invariant under a finite
  index subgroup $G_0<G$.  Since $G$ acts trivially on $I'$,
  $J\cap I'$ is $G$-invariant.  For each $i$, $J\cap I_i$ is
  $G_0$-invariant, and since by assumption $G_0$ acts irreducibly on
  $I_i$, we either get that $J\cap I_i$ is  empty or $J=I_i$, in
  particular, $J\cap I_i$ is $G$-invariant. Since this holds for each $i$, $J$ is $G$-invariant. 
\end{proof}

\begin{rem}\label{rem;rem1} 
If $\cald$ and $\cald'$ are two domains with an
interval exchange bijection $\tau$ from $\cald$ to $\cald'$, then a group
$G<\IET(\cald)$ is unfragmentable (resp. irreducible) if and only if its conjugate by $\tau$ in
$\IET(\cald')$ is unfragmentable (resp.  irreducible). 
\end{rem}

 The main theorem of this section is the following. 
We will use it only in the case of a cyclic group, but the general statement seems to be of interest.

\begin{thm}\label{thm;fi_unfragmentable}
If $G<\IET(\cald)$ is a  finitely generated subgroup of $\IET(\cald)$, it admits a
finite index subgroup that is unfragmentable. 
\end{thm}

Before proving the theorem, we  prove that one can find a finite symmetric generating set $S$ of $G$ such that 
for every $x\in \rond\cald$, $\Reg(x,S)=\Reg(x,G)$, and $E(S)=E(G)$.

\begin{lem}\label{lem_reg}
  Let $G$ be a finitely generated subgroup of $\IET(\cald)$.

Then there exists a finite symmetric generating set $S$ such that for all $x\in \rond\cald$,
$\Reg(x,S)=\Reg(x,G)$.
\end{lem}

\begin{proof}
If $x\in\rond \cald\setminus \Sing(G)$, then for all generating $S$ of
$G$,  $\Reg(x,S)=G x=\Reg(x,G)$.
Let $S$ be a finite symmetric generating set of $G$. 
We are going to increase $S$ so that the Lemma holds.
Fix a point $x\in \Sing(G)$.
Let $R_1,\dots,R_n$ be the partition of $\Reg(x,G)$ into $S$-regular orbits, and choose $x_i\in R_i$ for each $i\leq n$.
By definition of $\Reg(x,G)$, there exists $g_i\in G$ such that $g_i.x_1=x_i$ and $g_i$ is continuous at $x_1$.
Adding $\{g_2^{\pm1},\dots,g_n^{\pm 1}\}$ to $S$ yields a symmetric generating set $S'$ such that $\Reg(x,S')=\Reg(x,G)$.
Since $\Sing(G)$ is a union of finitely  many $G$-regular orbits,
one can repeat this operation finitely many times and get a generating set satisfying the lemma.
\end{proof}

Recall that $E(S)$ (resp. $E(G)$) is the set of points $x\in \Sing(G)$ whose
$S$-regular orbit $\Reg(x,S)$  (resp. whose $G$-regular orbit
$\Reg(x,G)$) is finite. The previous Lemma gives the following.

\begin{cor}\label{cor_Evide}
There exists a finite generating set $S$ of $G$ such that $E(G)=E(S)$.
\end{cor}

\begin{proof}[Proof of Theorem \ref{thm;fi_unfragmentable}]
  After cutting $\cald$ as in Remark \ref{rem;rem}, one can assume
that $E_S(G)=\es$.  Apply Imanishi theorem,
and write $\cald=\cald_{\fin}\cup\cald_\infty  $, with $\cald_{\infty}=I_1\cup \dots\cup I_r$ as in Corollary \ref{cor_Imanishi}.

Let $G_1$ be a finite index subgroup of
$G$ such that  $(G_1)_{|\cald_{\fin}}$ is trivial, and therefore
 the map $G\ra  (G_1)_{|\cald_{\infty}}$ is an isomorphism.  
Without loss of generality, and  to keep readable notations, we assume that $\cald=\cald_\infty$,
i.e.\ that every $G_1$-orbit is infinite.

Let $G_0$ be an arbitrary finite index subgroup of $G_1$.
The fact that $E_S(G)=\es$ means that for every $x\in \rond \cald_{\infty}$,
$\Reg(x,G)$ is infinite. We claim that $\Reg(x,G_0)$ is also infinite.
Indeed, let $g_i$ be a sequence of elements of $G$ that are continuous at $x$ and such that 
the points $g_i.x$ are all distinct. Since $[G:G_0]<\infty$, up to extracting a subsequence, 
we may assume that there exists an element $a\in G$ such that $ag_i\in
G_0$ for all $i$.
Since $\Disc(a)$ is finite, $a$ is continuous at $g_i.x$ for $i$ large
enough, so $ag_i.x\in \Reg(x,G_0)$. This proves our claim
and shows that $E(G_0)=\es$. 

By Corollary \ref{cor_Evide}, there exists a finite generating set $S_0$ of $G_0$
such that $E(S_0)=\es$.
Then Corollary \ref{cor_Imanishi}
yields a decomposition of $\cald$ into finitely many $G_0$-invariant subdomains $I_1,\dots,I_r$
on which the action of $G_0$ is irreducible, and since $E(S_0)=\es$, each $I_i$ is a union of connected components of $\cald$. 
The number $r$ depends on $G_0$ but is bounded by the number of connected components of $\cald$.

Among all possible choices of finite index subgroups $G_0<G_1$, we choose $G_0$ so that $r$ is maximal.
Then for any $G'_0<G_0$ of finite index, the decomposition of $\cald$
into $G'_0$-irreducible components  is
its decomposition into $G_0$-irreducible components.
This shows that $G_0$ is unfragmentable.
\end{proof}

\section{Commutation and solvable subgroups}

Let $\cald$ be a domain.
The main result of this section is the following.  

\begin{thm} \label{thm;tf_vA}
  Let $G<\IET(\cald)$ be a finitely generated torsion free solvable group. 
Then $G$ is virtually abelian.
\end{thm}

 Since virtually polycyclic groups are virtually torsion-free, we get

\begin{cor}\label{cor;VPC}
   Any virtually polycyclic subgroup of $\IET$ is virtually abelian.\qed
\end{cor}

The theorem will be proved in several steps.
We start with the following property of unfragmentable elements.

\begin{lem}\label{lem;if_a_unfragmentable}
Let $a\in \IET(\cald)$ be unfragmentable with irreducible components $\Irred(a)=\{I_1,\dots, I_r\}$. If $g\in \IET(\cald)$ is such that $gag^{-1}$ commutes with $a$,
then  for all $i$, either  $g(I_i)$ is  disjoint from $I_1,\dots, I_r$, 
or $g(I_i)$ is equal to some $I_j$.  
\end{lem}

\begin{proof}
Observe that $\Irred(gag^{-1}) =\{  g (I_1), \dots, g (I_r)\}$.
For readability we will write $a^g= gag^{-1}$. Since $a$
commutes with $a^g$, it permutes the collection $\{  g(I_1), \dots, g
(I_r)\}$, and therefore some power $a^{r!}$ preserves each   $g(I_j)$.
By unfragmentability, any subdomain preserved by $a^{r!}$ is preserved by $a$,
and therefore $a$ itself preserves each $g(I_j)$. 
Similarily, $a^g$ preserves each $I_j$.

The intersection   $g(I_j)\cap I_i$ is $a^g$-invariant, and also $a$ invariant. 
  If  $g(I_j)\cap I_i\neq\es$, then  
 $g(I_j)\cap I_i=I_i$ by  irreducibility of $a$ on $I_i$.
Similarly,  $g(I_j)\cap I_i=g(I_j)$ by irreducibility of $a^g$ on $g(I_j)$.
 It follows that if $g(I_j)\cap I_i\neq\es$ then $g(I_j)=I_i$.  The lemma follows.
\end{proof}

\begin{cor}\label{cor;dsbdom}  
 Let $a\in G<\IET(\cald)$ be a unfragmentable element with irreducible components $I_1,\dots,I_n$. 
Assume that for all $g\in G$, $gag\m$ commutes with $a$.

Then  
for each $i\leq n$,  $\{g(I_i), g\in G \}$ is a finite
collection of disjoint subdomains.  
\end{cor}

\begin{proof}
Let $g, h\in G$. Assume that $g(I_i)\cap h(I_i) \neq \es$. Then,
$h^{-1} g(I_i) \cap I_i
\neq \es$, and by  Lemma \ref{lem;if_a_unfragmentable},  this implies that $h^{-1} g(I_i) =I_i$,
hence $g(I_i)= h(I_i)$. Thus $\{g(I_i), g\in G \}$ is a 
collection of disjoint subdomains. It is finite because the measure of
$\cald$ is finite.
\end{proof}

\begin{prop} \label{prop;irred_in_A} Let $G$ be a subgroup of $\IET(\cald)$. 
  If there exists  a normal abelian subgroup  of $G$,  containing some
  irreducible and unfragmentable element,  
then $G$ is either abelian or virtually cyclic.
In particular, $G$ is virtually abelian.
\end{prop}

Before starting the proof, let us
recall the following, from \cite{DFG1}.  
 Let $g\in \IET(\cald)$  and $d(g)$ be the number of
discontinuity points of $g$ on $\cald$. Let $\|g \| = \lim \frac{1}{n}
d(g^n)$. In \cite[Coro. 2.5]{DFG1}, we proved that  $\|g \| =0$ if and
only if $g$ is 
conjugate to a continuous transformation of some
domain $\cald'$ consisting only of circles, and one of its powers is a
rotation on each circle.  

\begin{proof}
Let $a$ be such an element in $A \normal G$ (with $A$ abelian). 
If $\|a\|=0$, then as we mentionned in the preceeding discussion, $a$ is conjugate to a continuous transformation on some $\cald'$,
and because it is irreducible and unfragmentable, $\cald'$ has to be a circle and $a$ is an irrational rotation.
Then, by \cite[Lemma 1.1]{DFG1}, its centralizer is conjugate to the rotation group on $\cald'$,
and since $G$ normalises $A$,  
\cite[Lemma 1.1]{DFG1} says that $G$ must
also be     
 conjugate to a group of rotations on  $\cald'$, hence abelian.

Assume now that $\|a\|>0$.  Since $a$ is irreducible,  \cite[Prop. 1.5]{No} implies that
its centraliser is virtually cyclic, hence so is $A$.  
Therefore $A$ has a finite automorphism group. 
The group $G$ acts by conjugation on $A$, and since the automorphism group of $A$ is finite, the kernel of this action has finite index in $G$. 
This kernel is contained in the centralizer of $a$ though, so $G$ is
virtually cyclic. 
\end{proof}

\begin{prop}\label{prop;io_in_A} 
  If $G <\IET(\cald)$ is irreducible, and  contains a normal abelian
  subgroup  $A\normal G$,  with an element   $a\in A$ of infinite order,
then $G$ is virtually abelian.

In particular, if $G$ is finitely generated,  then so is $A$.
\end{prop}

\begin{proof}
 Up to replacing $a$ by a power, Theorem \ref{thm;fi_unfragmentable} allows us to assume that $a$ is unfragmentable.
  Let $I_1,\dots I_k$ be its irreducible components.

By Corollary \ref{cor;dsbdom},  $\{g (I_1), g\in G\}$ is a finite
collection of disjoint subdomains.  
 By irreducibility of $G$,
it is a partition of $\cald$, and we thus write $\cald=M_1\dunion\dots\dunion M_l$ with $M_1=I_1$, and $M_i=g_i I_1$.
There is a finite index subgroup $G_0$ of $G$ that preserves each $M_i$.
Let $k$ be such that $a^k\in G_0$. Then, $a^k$ is irreducible and
unfragmentable on $I_1$
so Proposition \ref{prop;irred_in_A} ensures that the image 
${G_0}_{|I_1}\subset \IET(I_1)$ of $G_0$ under the restriction map is virtually abelian. 
Similarly, $(a^k)^{g_i}$ is irreducible on $M_i$, so $G_0$ has virtually
abelian image in restriction to each $M_i$. 

 Since $\cald = M_1\dunion\dots\dunion M_l$, the product of restriction maps yields an injection
of $G_0$ into $\prod_j \IET (M_j)$. 
Thus, the image of $G_0$ is contained in a finite product of virtually abelian groups,
so $G_0$ is virtually abelian and so is $G$. 
\end{proof}

We finish with this lemma before proving Theorem \ref{thm;tf_vA}.

\begin{lem}\label{lem;lem_on_A_de}
 Let $G<\IET(\cald)$ be a finitely generated group, and assume  that  it
 contains a torsion-free  abelian normal subgroup $A$. Then $A$ is
 finitely generated, and there is  a finite
 index subgroup $H$ in $G$, such that $A \cap [H,H] = \{1\}$. 
\end{lem}

\begin{proof}
Let $\Irred(G) = \{I_1, \dots, I_k, I_{k+1}\dots, I_r\}$ where
one has ordered the components so that 
the image $A_{|I_i}$ of $A$  in $\IET(I_i)$ under the restriction map is a torsion group for all $i >k$ and contains an
infinite order element for $i\leq k$. Let $\cald_\fin=\cald \setminus
(\bigcup_{i=1}^r I_i)$. 

For every $i\leq k$,  by Proposition \ref{prop;io_in_A} we get that
the image $G_{|I_i}$ of $G$ in $\IET(I_i)$ under the restriction map is a virtually abelian group.

  In other components, the image of $A$ is a torsion group. Let us
  denote by $G_1$ the image of $G$ in $\IET(\bigcup_{i=1}^k I_i)$ and
  by $G_2$ the image of $G$ in $\IET(\left(\bigcup_{j=k+1}^r I_j\right) \cup \cald_\fin)$ under the restriction maps.

This gives an embedding $\iota : G\into G_1\times G_2$, where $G_1$ is virtually abelian,
and $p_2\circ \iota (A)$ is torsion. In particular, because $A$ is torsion-free, $p_1\circ \iota$ is
injective in restriction to $A$. 

 It already follows that $A$ is
finitely generated, since it embeds as a subgroup of a  finitely generated virtually
abelian group.

Consider $H_1$ an abelian finite index subgroup in $G_1$, and $H$  the
preimage of $H_1$ in $G$, which is a finite index subgroup.   

We saw that the map $p_1\circ \iota$ is injective on $A$, but it vanishes on
$[H,H]$ because $p_1\circ \iota (H) = H_1$ is abelian. Therefore $A
\cap [H,H]=\{1\}$, thus establishing the lemma.
\end{proof}

\begin{proof}[Proof of Theorem \ref{thm;tf_vA}]
Consider $G<\IET(\cald)$  a finitely
generated torsion free solvable group.  
 Then its
derived series has a largest index $n$ for which  $G^{(n)}\neq \{1\}$
(where
$G^{(n)} =[ G^{(n-1)},
G^{(n-1)}]$ and $G^{(0)}=G$). 
If $n=1$, $G$ is abelian, we thus proceed by induction on $n$.

 Since $G$  is torsion
free,  $G^{(n)}$  is torsion free, infinite, and it is also  abelian
and normal in $G$. We may then apply Lemma
\ref{lem;lem_on_A_de} to $A= G^{(n)}$, to find that there is a finite
index subgroup $H$ of $G$ such that $[H,H]\cap G^{(n)}=\{1\}$. However,
$H^{(n)} \subset G^{(n)}$, and it
follows that $H^{(n)}$ is trivial. The induction hypothesis implies
that $H$ is virtually abelian, hence so is $G$.  
\end{proof}

\section{Lamps and lighters}

\subsection{A lamplighter group in $\IET$}

\begin{figure}[ht]
\begin{center}
\includegraphics[scale=.4]{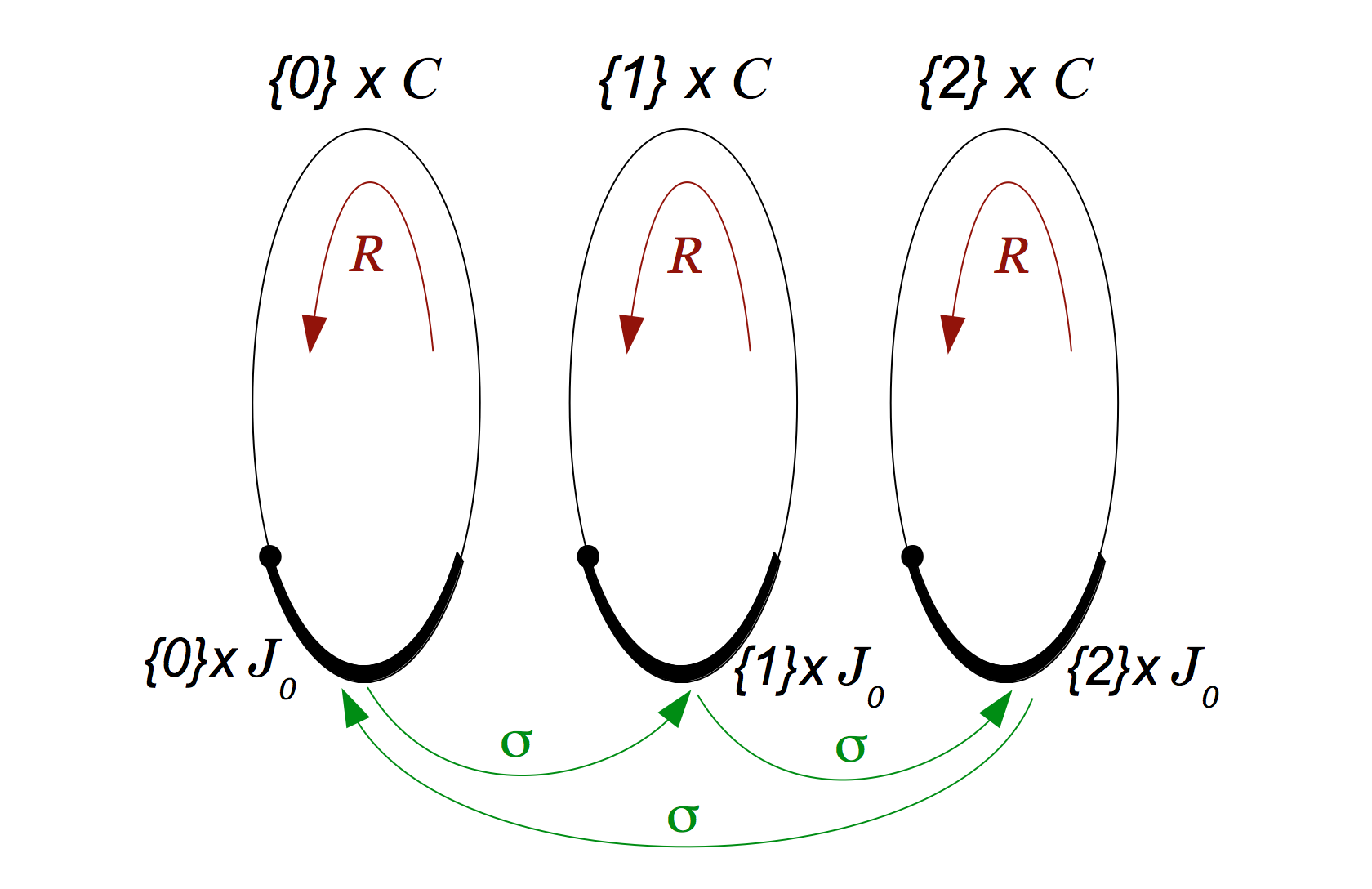}
\caption{A Lamplighter group $(\bbZ/3\bbZ)\wr \bbZ$ in $\IET$. 
The three circles $\{i\}\times \calc$ for $i\in \bbZ/3\bbZ$ are
visible. 
The transformation $R$ rotates each circle by the irrational angle
$\theta$. 
The support of the transformation $\sigma=\sigma_{1,J_0}$ is the union of the bold arcs.
}
\label{fig;1}
\end{center}
\end{figure}

\begin{prop}\label{prop_abelianLL}
  For all finite abelian group $A$,  the
group $A\wr\bbZ$ embeds in $\IET$.
\end{prop}

 Recall that in general, the group $A\wr G$ is the group 
$(\oplus_{i\in G} A)\isemidirect G$,
where $G$ acts by shifting coordinates: if $g\in G$,
and $(a_i)_{i\in G} \in (\oplus_{i\in G} A) $ is an almost null sequence, 
$g.(a_i)_{i\in G}=(a_{g\m i})_{i\in G}$.

 We will actually describe an embedding in $\IET(\cald)$ for a certain
 domain $\cald$. However, $\IET$ and $\IET(\cald)$ are isomorphic, and
 the choice of $\cald$ is only for convenience.

The construction is
illustrated in Figure \ref{fig;1}.

\begin{proof}
  Consider the domain $\cald=A\times \calc$ where $\calc=\bbR/\bbZ$.
Given $a\in A$ and $J\subset \calc$ a subinterval, let 
$\sigma_{a,J}$ 
be the element of $\IET(\cald)$ defined for all $(a',x)\in A\times \calc$ by
$$\sigma_{a,J}.(a',x)=
\begin{cases}
(aa',x)&\text{ if $x\in J$}\\
(a',x)& \text{ if $x\notin J.$}
\end{cases}
$$
Note that the support of $\sigma_{a,J}$ is $A\times J$ when $a\neq 1$.
We define $\cala_{J}$ as the subgroup of $\IET(\cald)$
consisting of the elements $\sigma_{a,J}$ for $a\in A$ (note that $\cala_{J}$ is isomorphic to $A$
as long as $J$ is non-empty).

Fix $J_0=[0,1/2[\subset \calc$, and let $\cala=\cala_{J_0}$.
Let $\theta\in \bbR\setminus \bbQ$, and let $R\in \IET(\cald)$ be the rotation by $\theta$
on each circle: $R(a,x)=(a,x+\theta)$.
We claim that the group generated by $R$ and $\cala_J$ is isomorphic to $A\wr\bbZ$.

First, one easily checks that for all  $J, J'\subset \calc$, 
any element of $\cala_J$ commutes with any element of $\cala_{J'}$
(this is because $A$ is abelian).
 Denote by $t$ a generator of the factor $\bbZ$ in $A\wr \bbZ$.
Since $R^k\cala R^{-k}=\cala_{R^k J_0}$, 
there is a homomorphism $\phi:A\wr\bbZ\ra \grp{R,\cala}$
sending  $t$ to $R$ and sending the almost null sequence $(a_i)_{i\in \bbZ}$
to $\prod_{i\in\bbZ} \sigma_{a_i,R^i J_0}$.

To prove that $\phi$ is injective, consider an element  $g=((a_i)_{i\in\bbZ},t^k)$ of its kernel.
Since $\phi(g)$ sends $(a,x)$ to some $(a',x+k\theta)$, we get that $k=0$.
This means that $\phi(g)$ is a commuting product $\phi(g)=\sigma_{a_1,J_1}\dots \sigma_{a_n,J_n}$ 
where $J_1,\dots J_n$ are distinct translates of $J_0$ by a multiple of $\theta$.
We can assume that $n>0$ and that no $a_i$ is trivial.
For all $x\in \calc$, consider $\calj_x\subset\{1,\dots,n\}$ the set of indices $j\in \{1,\dots,n\}$ 
such that $x\in J_j$.
Note that when $x$ varies, $\calj_x$ changes only when $x$ crosses
 an endpoint of some $J_i$, and that $\calj_x$ then changes by exactly one element
(this is because the $2n$ endpoints of the $J_i$'s are distinct since $\theta$ is irrational).
Thus, there exist  $x,x'\in \calc$ and $i_0\in\{1,\dots n\}$
such that $\calj_{x'}=\calj_x\cup\{J_{i_0}\}$. 
Now, for all $a\in A$, $\phi(g) (a,x)=(a \prod_{j\in \calj_x} a_j,x)$,
and since $\phi(g)$ is the identity,
$\prod_{j\in \calj_x} a_j=1$.
Similarly, $\prod_{j\in \calj_{x'}} a_j=1$, so $a_{i_0}=1$, a contradiction.
\end{proof}

The argument above  immediately generalizes to $A\wr \bbZ^d$, by replacing the
rotation of angle $\theta$ by $d$ rotations of rationally independant angles.  We thus get:

\begin{prop}\label{prop;LLk}
  For all $d\geq 1$, and all finite abelian group $A$, $A\wr \bbZ^d$ embeds in $\IET$.
\end{prop}

\begin{rem}    
  Given $G<\IET$, and $A$ a finite abelian group, 
we don't know when the group $A\wr G$ embeds in
$\IET$.  
\end{rem}

\subsection{Lamps must commute}

We now put restrictions on which wreath products may embed in $\IET$.
Note that if a group $A$ contains an infinite order element, then $A\wr \bbZ$ contains
the torsion-free solvable group $\bbZ\wr\bbZ$ which does not embed in $\IET$ by Theorem \ref{thm;tf_vA}.

\begin{thm}\label{thm_LLNA}
  Let $L=F\wr \bbZ$ with $F$ finite non-abelian.
Then $L$ does not embed in $\IET$.
\end{thm}

We will use several times the observation  (see \cite[Lemma
6.2]{DFG1}) that for any finitely
generated subgroup  $G$ of  $\IET(\cald)$, the orbit of any point of the domain $\cald$ has
polynomial growth  in the following sense: 
given a finite generating set and the corresponding word metric  on
$G$,  
denoting by $B_R$ the ball of radius $R$ in $G$,
there exists a polynomial $P$
such that for all $x\in \cald$ and all $R\geq 0$, $\# (B_R.x)\leq P(R)$.

The theorem will be proved by showing that if $L$ did embed in $\IET$, there would exist
an orbit with exponential growth.

If $E\subset [0,1)$ has positive measure and if $T$ is any element of $\IET$, then by
Poincar\'e recurrence Theorem, the orbit of almost every $x\in E$ comes back to $E$.
We will need a more precise estimate, that comes from Birkhoff theorem.

\begin{lem}\label{lem;Birk}
Let $E\subset [0,1)$ have positive measure, and let $T$ be any element
in $\IET$.

Then there exists $x_0\in [0,1)$, and constants $\alpha >0,\, \beta>0$,  
such that for all $n$,   $$\#\{i\,|\ 0\leq i< n, T^i(x)\in E\}\geq \alpha n - \beta.$$
\end{lem}

\begin{proof}
Assume first that  the Lebesgue measure $\mu$ is ergodic, and apply Birkhoff theorem
to the characteristic function $f$ of $E$.
Then for almost every $x$, $\frac1n \sum_{i=0}^{n-1} f(T^i(x))\ra\mu(E)$, 
so for $n$ large enough, 
$\frac1n \sum_{i=0}^{n-1} f(T^i(x))\geq\frac12\mu(E)$, 
so $\#\{i| 0\leq i\leq n, T^i(x)\in E\}\geq \frac12 \mu(E) n $ for $n$ large enough,
and the result holds.

If  the Lebesgue measure is not ergodic, Birkhoff    
 theorem still says that the limit
exists almost everywhere, and the limit is a $T$-invariant function $l(x)$ having
the same average as $f$, that is $\int l(x)dx=\mu(E)$  (see \cite{KW} for instance). It follows that there exist points where $l(x)\geq \mu(E)/2$,
and the proof works the same.  
\end{proof}

We prove an abstract algebraic lemma that relates the stabilizer of a point in a product
of groups to coordinate-wise stabilizers.

\begin{lem}\label{lem;algeb}
 Consider some groups $F_i$, and an action of $F=F_1\times \dots\times F_n$ on a set $X$.
Let $x\in X$, $\Stab(x)$ its stabilizer in $F$, and $S_i=F_i\cap \Stab(x)$. 
Let $N_i$ the normalizer of $S_i$ in $F_i$.

Then $\Stab(x)\subset N_1\times \dots\times N_n$.

In particular, if no $S_i$ is normal in $F_i$,
then $\# F.x\geq 2^n$.
\end{lem}

\begin{proof} 
Note that $N_1\times \dots\times N_n=\cap_i N_{F}(S_i)$,
so we have to prove that for each $i$, any $g\in \Stab(x)$ normalizes $S_i$.
If $g\in \Stab(x)$, 
then $g$ normalizes both $\Stab(x)$ and $F_i$ (because $F_i\normal F$),
hence normalizes their intersection, namely $S_i$.

Let us prove the last comment.  
The given assumption says that for all $i$,  $[
F_i:N_i]\geq 2$ so $\#F.x=[F:\Stab(x)]\geq 2^n$.
\end{proof}

 To prove the theorem, consider $L=F_0\wr \bbZ$ for some finite non-abelian group $F_0$.
We assume that $L$ embeds in $\IET$ and argue towards a contradiction. 
Let $t$ be a generator of $\bbZ$ viewed as a subgroup of $L$, and write $F_i=F_0^{t^i}$.  
We equip the group $G$ with the word metric corresponding to the
generating set $\{t\}\cup F_0$.  
The subgroup  $F_0 \times F_1\times \dots \times F_{n-1}$ is contained in the ball of radius $3n$.
We identify $L$ with the corresponding subgroup of $\IET$.
Since orbits grow polynomially \cite[Lemma 6.2]{DFG1}, for any  $x\in[0,1)$, the orbit of $x$ under $F_0\times\dots \times F_{n-1}$
has to be bounded by a polynomial in $n$.

\begin{prop}\label{prop_normal}
  For all $x\in [0,1)$,  $\Stab_{F_0}(x)\normal F_0$.
\end{prop}

\begin{proof}
  Otherwise, let $E\subset [0,1)$ be the set of points where
$\Stab_{F_0}(x)\not\normal F_0$. Since $F_0$ is a finite group, this is a subdomain of $[0,1)$,
and it has positive measure.
We apply Birkhoff theorem (in the form of Lemma \ref{lem;Birk}), and get that there exists $x\in[0,1)$, $\alpha,\beta>0$, such that,
for all $n$, there exists $k_n\geq \alpha n-\beta$, and some indices
$0\leq i_1<\dots< i_{k_n}\leq n$ such that
$t^{i_j}(x)\in E$. Applying the algebraic Lemma \ref{lem;algeb} to $F=F_{i_1}\times \dots\times F_{i_{k_n}}$ we get that 
$\# (F.x)\geq 2^{k_n}\geq 2^{\alpha n-\beta}$.
Since $F$ is contained in a ball of linear radius in $L$, this contradicts polynomial growth of orbits.
\end{proof}

We now prove another algebraic lemma.

\begin{lem}\label{lem;algeb2}
Consider an action of $F=F_1\times \dots \times F_n$ on a set $X$.
Let $x\in X$, $\Stab(x)$ its stabilizer in $x$, and $S_i=\Stab(x)\cap F_i$ its stabilizer in $F_i$.

Assume that $\Stab(x)\normal F$ (in particular $S_i\normal F_i$).
Consider  $Z(F_i/S_i)$ the center of the quotient group $F_i/S_i$,
and $Q_i=(F_i/S_i)/Z(F_i/S_i)$.

Then the natural epimorphism $F_1\times \dots\times F_n\ra Q_1\times \dots\times Q_n$ 
factors through an epimorphism $F/\Stab(x)\onto Q_1\times \dots\times Q_n$.

In particular, if all groups $F_i/S_i$ are non-abelian, each $Q_i$ is non-trivial,
and $\#F.x=\#F/\Stab(x)\geq 2^n$.
\end{lem}

\begin{proof}
Consider $g=(g_1,\dots,g_n)\in \Stab(x)$, and denote by $\bar g_i$ the image of $g_i\in F_i/S_i$.
 We have to prove that    
$\bar g_i$ is central in $F_i/S_i$, in other words that for all $a\in F_i$, 
 $[g_i,a]\in S_i$. Since        
$S_i=\Stab(x)\cap F_i$ is normal in $F$, $[g_i,a]=g_i(ag_i\m a\m)\in S_i$.
The lemma follows.
\end{proof}

\begin{proof}[Proof of Theorem \ref{thm_LLNA}]
Fix $k\geq 1$ and consider the subgroup $L'\subset L$ generated by $t^k$, and $F_0\times\dots\times F_{k-1}$.
Clearly, $L'\simeq F'_0\wr \bbZ $ where $F'_0=F_0\times \dots\times F_{k-1}$.
Applying Proposition \ref{prop_normal} to $L'$, we get that for all $x$, and all $k$, the stabilizer of $x$
in $F_0\times\dots\times F_{k-1}$ is normal in $F_0\times\dots\times F_{k-1}$.
 
Let $E\subset [0,1)$ be the set of points $x$ where $F_0/\Stab(x)$ is non-abelian.
If $E$ is empty, any commutator of $F_0$ acts trivially in $[0,1)$, so $F_0$ does not act faithfully, a contradiction.
So $E$ is a  non-empty subdomain and has positive measure.
By Lemma \ref{lem;Birk},  there exists $x\in[0,1)$, $\alpha,\beta>0$, such that,
for all $n$, there exists $k_n\geq \alpha n-\beta$, and some indices
$0\leq i_1<\dots< i_{k_n}\leq n$ such that
$t^{i_j}(x)\in E$.

Applying the second algebraic Lemma \ref{lem;algeb2} to $F=F_{i_1}\times \dots\times F_{i_{k_n}}$ we get that 
$\# F.x\geq 2^{\alpha n-\beta} $.
This contradicts polynomial growth of orbits as above. 
\end{proof}

\subsection{Uncountably many solvable groups, via  lamplighter-lighters}

We finish by an abundance result among solvable subgroups of $\IET$
(necessarily with a lot of torsion in view of Theorem \ref{thm;tf_vA}).

\begin{thm}\label{thm_uncountable}
  There exist uncountably many isomorphism classes of 3-generated solvable subgroups of derived length $3$
in $\IET$.
\end{thm}

These will be obtained as iterated lamplighter-like
constructions. Again we will be free to choose the domain $\cald$ for convenience.

\newcommand{\un}{{\mathbbm{1}}}

For the time being, take $\cald$ an arbitrary domain. Let $G<IET(\cald)$ and $(A,+)$ a finite abelian group.
The following construction generalizes the construction showing that lamplighter groups embed.
In this case, $G$ would be the cyclic group generated by an irrational rotation on a circle $\cald$.

Let $\calf=A^\cald$ be the additive group of all functions on $\cald$ with values in $A$.
Given $a\in A$ and $J\subset\cald$, we denote by $a\un_J\in \calf$ the function
defined by $a\un_J(x)=a$ if $x\in J$, and $a\un_J(x)=0$ otherwise.
The group $G$ acts on $\calf$ by precomposition.
Given a subdomain $J\subset \cald$, let $\calf_J\subset \calf$ be
the smallest $G$-invariant additive subgroup containing the functions $a\un_J$ for $a\in A$.

\begin{prop}(Lamplighter-like construction)

  Let $G$ be a subgroup of $\IET(\cald)$, $A$ a finite abelian group,  
 and $J\subset \cald$. 
Then    
  $\calf_J\isemidirect G$ embeds in $\IET$.
\end{prop}

\begin{rem}
  In the particular case of Proposition  \ref{prop_abelianLL}, we proved that $\calf_J$ is isomorphic to $\oplus_{g\in G}A$.
In general, there is a natural morphism from $\oplus_{g\in G}A$ to $\calf_J$, but it might be non-injective.
One easily constructs examples using a group $G$ that does not act freely on $\cald$.
\end{rem}

\begin{proof}
  Take $\cald'=A\times \cald$.
Embed $G$ in $\IET(\cald')$ by
setting for all $g\in G$,  $g(a,x)=(a,g(x))$.

For all $b\in A$, consider $\sigma_b\in IET(\cald')$ defined by
$\sigma_b(a,x)=(b+a,x)$ if $x\in J$,
and by $\sigma_b(a,x)=(a,x)$ if $x\notin J$.

We claim that the subgroup $G'$ generated by $G$ and $\{\sigma_b|b\in A\}$
is isomorphic to $ \calf_J\isemidirect G$.

Consider  the subset $H\subset IET(\cald')$ of all $T\in IET(\cald')$ such that 
there exists some transformation $g\in G$ and some $f\in \calf_J$ such that 
$T(a,x)=(a+f(x),g(x))$.
It is clear that $H$ is a subgroup of $\IET(\cald')$, that $H$ contains $G'$,
and that $G'=H$.
Finally, the fact that $H\simeq \calf_J\isemidirect G$ is also clear.
\end{proof}

\begin{figure}[ht]
\begin{center}
\includegraphics[scale=.4]{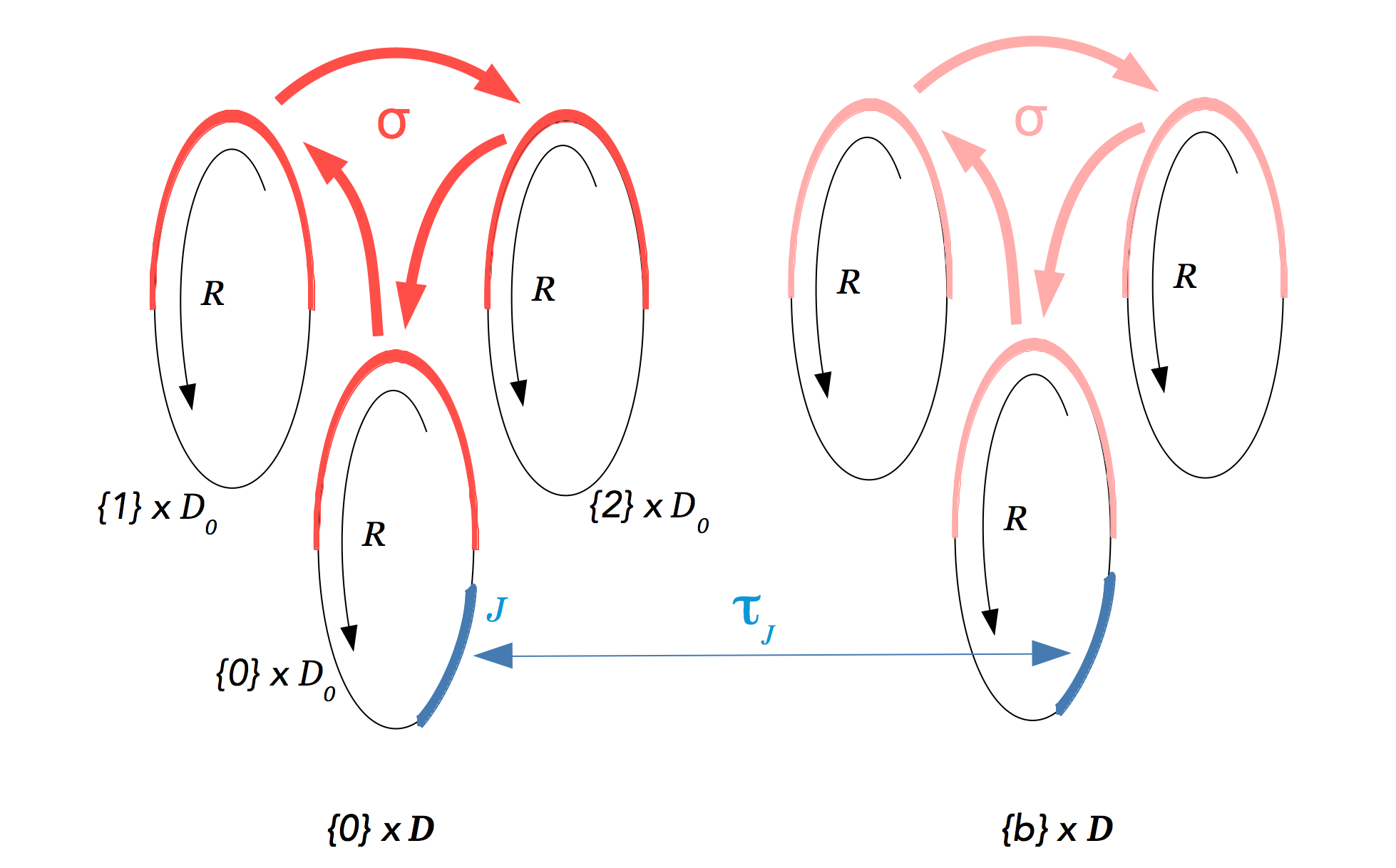}
\caption{A Lamplighter-like construction in $\IET$ producing the
  examples of Theorem \ref{thm_uncountable}: on the left 
, the domain $\cald= \mathbb{Z}/3\mathbb{Z} \times \cald_0$, and the
Lamplighter group $G =  \mathbb{Z}/3\mathbb{Z}  \wr \mathbb{Z}$
generated by a cyclic swap on the three intervals $\mathbb{Z}/3\mathbb{Z} \times I$ (in
bold red) and by
the simultaneous rotation $R$. On the right, $\{b\} \times \cald$  is the duplicated copy
of
$\cald$, with the swap $\tau_J$ on $J$ and its copy.  The group $G$ acts diagonally on
$\cald' =    \mathbb{Z}/2\mathbb{Z} \times \cald$. The group
$\displaystyle \left( \oplus_{G}  \mathbb{Z}/2\mathbb{Z} \right)
\isemidirect G$ acts on $\cald'$ but its image $H_J$ in $\IET(\cald')$
depends on $J$.
}
\label{fig;2}
\end{center}
\end{figure}

Let $\cald_0=\bbR/\bbZ$,   and $I= [0,1/2[\subset \cald_0$.  Fix once and
for all   $\alpha\notin \bbQ$, $\rho_\alpha$ the rotation
of $\cald_0$ of angle $\alpha$.  Let $G$ be the lamplighter
group  $(\bbZ/3\bbZ) \wr \bbZ$ realised as a subgroup of $\IET(\cald)$
for $\cald = ( \bbZ/3\bbZ)\times \cald_0 $ as in Proposition
\ref{prop_abelianLL}. 
We denote by $R$ the rotation $(a,x)\mapsto (a, \rho_\alpha(x))$ on $\cald$.

We now perform this lamplighter-like construction a second time,
starting with $G$ acting on $\cald$,  with $A=\bbZ/2\bbZ$
and with some $J\subset  \{0\}\times \cald_0 \subset \cald$.  We recall
that  $\calf_{J}$ is the   subgroup of the group of functions from
$\cald$ to $A=\bbZ/2\bbZ$ as above.   Define $H_{J}=
   \calf_{J} \isemidirect G$ the lamplighter-like group
   thus obtained, on $\cald'=(\bbZ/2\bbZ) \times \cald$. See Figure \ref{fig;2}.

\begin{prop}\label{prop_modQ}
Assume that $J_1,J_2\subset \{0\}\times \cald_0 \subset \cald$,
and that $|J_1|,|J_2|<\frac12$.

If $H_{J_1}\simeq H_{J_2}$   
then $|J_1|\in \mathrm{Vect}_\bbQ(1,\alpha,a_2,b_2)$, where $a_2,b_2\in\bbR/\bbZ$ are the endpoints of $J_2$.
\end{prop}

Denote by     
$\cale_I$ and $\cale_{J_i}\subset \bbR/\bbZ$ the orbit under the rotation $\rho_\alpha$ of the endpoints of $I$ and $J_i$ respectively
(ie $\cale_I=\{0,\frac12\}+\alpha\bbZ$, $\cale_{J_i}=\{a_i,b_i\}+\alpha\bbZ$). 

\begin{lem}\label{lem_main}
If  $H_{J_1}\simeq H_{J_2}$, then there exist   two subdomains $K,K'\subset \bbR/\bbZ$
with endpoints in $\cale_{J_2}$ and $\cale_I$ respectively and
$\eps\in\{\pm 1\}$  
such that $$\forall n\in\bbZ,\quad \rho_\alpha^n(I)\cap J_1\neq \es \Longleftrightarrow  \rho_\alpha^{\eps n}(K)\cap K'\neq \es.$$
\end{lem}

\begin{proof}[Proof of Proposition \ref{prop_modQ} from Lemma \ref{lem_main}.]
We have that $\rho_\alpha^n(I)\cap J_1\neq \es$ if and only if $n\alpha\in J_1-I$,
and that $\rho_\alpha^{\eps n}(K)\cap K'\neq \es$ if and only if $n\alpha\in \eps(K'-K)$.

By Birkhoff theorem, we get that $|J_1-I|=|K'-K|$.
Since $|J_1|<\frac 12=|I|$, $|J_1-I|=\frac12 +|J_1|$.

On the other hand, since  $K'-K$ is a union of intervals whose endpoints are in 
$\cale_{I}-\cale_{J_2}\subset \mathrm{Vect}_\bbQ(1,\alpha,a_2,b_2)$,
the proposition follows.
\end{proof}

We view $H_J$ as the semidirect product $\calf_J\isemidirect G$.
Thus each element $h\in H_J$ can  be written (uniquely) under the form
$h=g\tau$ with $\tau\in\calf_J$, $g\in G$.
Now we view $G$ as a group of interval exchange transformations on $\cald=(\bbZ/3\bbZ)\times (\bbR/\bbZ)$.
It is generated by the rotation $R:(a,x)\mapsto (a,x+\alpha)$, and 
by the lamp element $\sigma\in \cals$ (of order 3) that sends
$(a,x)\in (\bbZ/3\bbZ)\times (\bbR/\bbZ)$ to $(a+1,x)$ for $x\in I$, and
is the identity otherwise (see left part of figure \ref{fig;2}). 
Since $G$ is a lamplighter group, any $g\in G$ can be written uniquely
as  $R^nS_f$
where $S_f$ is an IET of the form
$(a,x)\ra (f(x)+a,x)$ for some function $f:\bbR/\bbZ\ra \bbZ/3\bbZ$.
We denote by $\cals=\{S_f\}$ the $3$-torsion abelian group of lamps of $G$.
It is freely generated by the $\grp{R}$-conjugates of $\sigma$. 
Thus any $h\in H_J$ is written in a unique way as $h=R^nS_f \tau$ as above.

The kernel of the natural map $H_J\ra \bbZ$ is the torsion group $N=\calf_J\cals$.
It is exactly the set of elements of finite order.

Denote by $b$ be the generator of $\bbZ/2\bbZ$. Given $K\subset \cald$
denote by  $\tau_K=b\un_{K}\in \calf$. For certain $K$ (for instance
  $K=J$), $\tau_K\in \calf_J$ and thus is in $H_J$.

\begin{fact}\label{fact1} 
If $K\subset \{0\}\times \cald_0$, then
  $[\sigma,\tau_K]=1$ if and only if $K\cap I=\es$.
\end{fact}

\begin{proof}
  From the definition of the semidirect product,  $\sigma\tau_K\sigma\m=  b\un_{\sigma(K)}$.
Hence $[\sigma,\tau_K]=1$ if and only if $\sigma(K)=K$.
If $K\cap I\neq \es$, then $\sigma(K)$ contains a point outside  $\{0\} \times \cald_0$, so $\sigma(K)\neq K$.
\end{proof}

\begin{fact}\label{fact2}
  For any $\tau,\tau'\in \calf_J$ and any $h\in H_J$, $[h\tau,\tau']=[h,\tau']$.
\end{fact}

\begin{proof}
This is is an immediate consequence of the  commutation of $\tau$ with $\tau'$.
\end{proof}

\begin{proof}[Proof of Lemma \ref{lem_main}]
 Let $\phi:H_{J_1}\ra H_{J_2}$ be an isomorphism.
We first note that $\calf_{J}\normal H_J$ is precisely the set of elements $g\in H_J$ such that $g^2=1$. 
Moreover, the set of elements of finite order in $H_J$ is the subgroup
$N=\calf_J\cals$    
(but there are exotic      elements of order $3$).

Thus $\phi(\tau_{J_1})\in\calf_{J_2}$ can be viewed as a function $\tau:\cald\ra \bbZ/2\bbZ$,
and $\phi(\sigma)$ as an element $S'\tau'$ for some $S'\in \cals$, and
$\tau'\in\calf_J$. 
Now $\phi(R)$ generates $G_{J_2}$ modulo the torsion subgroup, so
$\calf_{J_2}$, so $\phi(R)=R^{\eps}S''\tau''$
for some function $\tau''\in \calf_{J_2}$, $S''\in\cals$, and $\eps=\pm 1$.

In what follows, we use the notation $g^h$ for $hgh\m$. 
Fix $n\in\bbZ$ and consider the commutator $C=[\sigma,\tau_{J_1}^{R^n}]=[\sigma,\tau_{R^n(J_1)}]$. 
By Fact \ref{fact1}, $C$ is trivial if 
and only if $R^n(J_1)\cap I=\es$.
On the other hand, by Fact \ref{fact2}, $$\phi(C)=[S'\tau',\tau^{(R^{\eps}S''\tau'')^n}]=
[S',\tau^{(R^{\eps}S''\tau'')^n}].$$
There exists $\tau'''\in\calf_{J_2}$ such that $(R^{\eps}S''\tau'')^n=(R^{\eps}S'')^n\tau'''$,
and since $\calf_{J_2}$ is abelian, $\tau^{(R^{\eps}S''\tau'')^n}=\tau^{(R^{\eps}S'')^n}$.
Hence, $\phi(C)=1$ if and only if
$[S'^{(R^{\eps}S'')^{-n}},\tau]=1$.
A similar calculation in $G$ shows that $S'^{(R^{\eps}S'')^{-n}}=S'^{R^{-\eps n}S'''}=S'^{R^{-\eps n}}$,
so $\phi(C)=1$ if and only if $[S',\tau^{R^{\eps n}}]=1$.

Given $x\in \bbR/\bbZ$,   
we call the \emph{fiber} of $x$ the 3 point set $F_x=\{(a,x)|a\in \bbZ/3\bbZ\}\subset \cald$.
Let $K\subset \bbR/\bbZ$ be the set of points $x$ such that $\tau$ is non constant on the fiber of $x$.
Since $\tau\in \calf_{J_2}$, $K$ is a union of intervals with endpoints in $\cale_{J_2}$.
Write the transformation $S'$ as $(a,x)\mapsto (a+f(x),x)$ for some function $f:\bbR/\bbZ\ra \bbZ/3\bbZ$,
and let $K'$ be the support of $f$. It is a union of intervals with endpoints in $\cale_{I}$.

We claim that $[S',\tau^{R^{\eps n}}]=1$ if and only if $R^{\eps n}(K)\cap K'=\es$.
Note that $S'$  
preserves each fiber.
 We view $\tau$ as an element in the additive group $\calf$ of functions on $\cald$ with values in $\bbZ/2\bbZ$,
and $S',R$ as transformations of $\cald$, so that $[S',\tau^{R^{\eps n}}]=1$
if and only if $\tau\circ R^{\eps n} \circ S'=\tau\circ R^{\eps n}$.
So fix $x\in\bbR/\bbZ$, and   
$F_x$ its fiber.
If $x\notin K'$, then $S'$ acts as the identity on $F_x$, so 
 $\tau\circ R^{\eps n} \circ S'=\tau\circ R^{\eps n}$ in retriction to $F_x$.
If $x\notin R^{\eps n}(K)$, then $\tau^{R^{\eps n}}$ is a constant function on $F_x$, and the same conclusion holds.
If on the contrary $x\in R^{\eps n}(K)\cap K$, 
then $\tau^{R^{\eps n}}$ is not constant on $F_x$   
and since $S'$ acts transitively on $F_x$,
$\tau^{R^{\eps n}}\circ S'$ does not coincide with $\tau^{R^{\eps n}}$ on $F_x$.
This proves the claim and concludes the proof. 
\end{proof}

We can finally prove Theorem \ref{thm_uncountable}.
\begin{proof}[Proof of Theorem \ref{thm_uncountable}] 
Let $J_x=[0,x[$ for $x<\frac12$.
Proposition \ref{prop_modQ} shows that given $x$, there are at most countably many $y<\frac12$ such that
$H_{J_x}\simeq H_{J_y}$. This gives uncountably many isomorphism classes of groups.

Since $G$ is metabelian, $H_{J_x}$ is solvable of derived length at most $3$. 
Since there are only countably many isomorphism classes of finitely generated metabelian groups,
$H_{J_x}$ has derived length exactly $3$ for uncountably many values of $x$.
\end{proof}

\section*{Acknowledgement}
We would like to thank the anonymous referee for suggestions that improved the exposition.
The first and third authors acknowledge support from ANR-11-BS01-013
while this work was started, and are supported by the Institut Universitaire de
France.  The second named author is supported in part by Grant-in-Aid for Scientific Research (No. 15H05739). 
The author thanks the Centre Henri Lebesgue ANR-11-LABX-0020-01 for creating an attractive mathematical environment.
This material is based upon work supported by the National Science
Fondation under grant No. DMS-1440140 while the authors were in
residence at the Mathematical Science Research Institute in Berkeley
California, during the Fall 2016 semester.

\vskip .3cm
{\sc Fran\c{c}ois Dahmani}, Universit\'e Grenoble Alpes, Institut
Fourier, F-38000 Grenoble, France. {\tt francois.dahmani@univ-grenoble-alpes.fr}

\vskip .2cm

{\sc Koji Fujiwara}, Department of Mathematics, Kyoto University,
Kyoto, 606-8502 Japan  {\tt kfujiwara@math.kyoto-u.ac.jp }

\vskip .2cm

{\sc Vincent Guirardel}, Univ Rennes, CNRS, IRMAR - UMR 6625, F-35000 Rennes, France.
{\tt vincent.guirardel@univ-rennes1.fr}

\end{document}